\documentclass{amsart}

\usepackage{amssymb}
\usepackage{amsthm}
\usepackage{amsmath}

\usepackage[normalem]{ulem}

\usepackage[shortlabels]{enumitem}

\usepackage{hyperref}

\theoremstyle{plain}
\newtheorem{proposition}{Proposition}[section]
\newtheorem{theorem}[proposition]{Theorem}
\newtheorem{lemma}[proposition]{Lemma}

\theoremstyle{definition}
\newtheorem{example}[proposition]{Example}
\newtheorem{definition}[proposition]{Definition}
\newtheorem{observation}[proposition]{Observation}
\theoremstyle{remark}
\newtheorem{remark}[proposition]{Remark}

\DeclareMathOperator{\Aut}{Aut}

\DeclareMathOperator{\GL}{GL}

\DeclareMathOperator{\PSL}{PSL}
\DeclareMathOperator{\PGL}{PGL}

\DeclareMathOperator{\End}{End}

\DeclareMathOperator{\Spanset}{Span} 
\DeclareMathOperator{\Gr}{Gr} 
\DeclareMathOperator{\id}{id} 
\DeclareMathOperator{\Haus}{Haus} 
\DeclareMathOperator{\CAT}{CAT} 
\DeclareMathOperator{\Isom}{Isom}

\DeclareMathOperator{\Min}{Min}
\DeclareMathOperator{\Stab}{Stab}

\DeclareMathOperator{\CH}{ConvHull}
\DeclareMathOperator{\relint}{rel-int}

\DeclareMathOperator{\partiali}{\partial_{\, i}}
\DeclareMathOperator{\partialni}{\partial_{\, n}}

\newcommand{\hil}{H_{\Omega}}

\DeclareMathOperator{\Cc}{\mathcal{C}}
\DeclareMathOperator{\Dc}{\mathcal{D}}

\DeclareMathOperator{\Kc}{\mathcal{K}}
\DeclareMathOperator{\Lc}{\mathcal{L}}
\DeclareMathOperator{\Nc}{\mathcal{N}}

\DeclareMathOperator{\Cb}{\mathbb{C}}

\DeclareMathOperator{\Pb}{\mathbb{P}}
\DeclareMathOperator{\Rb}{\mathbb{R}}

\DeclareMathOperator{\Zb}{\mathbb{Z}}

\newcommand{\abs}[1]{\left|#1\right|}

\newcommand{\norm}[1]{\left\|#1\right\|}

\newcommand{\ip}[1]{\left\langle #1\right\rangle}

\begin{document}

\title[A flat torus theorem for convex co-compact actions]{A flat torus theorem for convex co-compact actions of projective linear groups}
\author{Mitul Islam}\address{Department of Mathematics, University of Michigan, Ann Arbor, MI, USA}
\email{mitulmi@umich.edu}
\author{Andrew Zimmer}\address{Department of Mathematics, Louisiana State University, Baton Rouge, LA, USA}
\curraddr{Department of Mathematics, University of Wisconsin-Madison, Madison, WI, USA}
\email{amzimmer2@wisc.edu}
\date{\today}

\begin{abstract} 
In this paper we consider discrete groups in $\PGL_d(\Rb)$ acting convex co-compactly on a properly convex domain in real projective space. For such groups, we establish an analogue of the well-known flat torus theorem for $\CAT(0)$ spaces. 
\end{abstract}

\maketitle

\section{Introduction} 

If $G$ is a connected simple Lie group with trivial center and $K \leq G$ is a maximal compact subgroup, then $X=G/K$ has a unique (up to scaling) Riemannian symmetric metric $g$ such that $G = \Isom_0(X,g)$. The metric $g$ is non-positively curved and $X$ is simply connected, hence every two points in $X$ are joined by a unique geodesic segment. A subset $\Cc \subset X$ is called \emph{convex} if for every $x,y \in \Cc$ the geodesic joining them is also in $\Cc$.  Finally, a discrete group $\Gamma \leq G$ is said to be \emph{convex co-compact} if there exists a non-empty closed convex set $\Cc \subset X$ such that $\gamma(\Cc) = \Cc$ for all $\gamma \in \Gamma$ and the quotient $\Gamma \backslash \Cc$ is compact.

In the case in which $G$ has real rank one, for instance $G = \PSL_2(\Rb)$, there are an abundance of examples of convex co-compact subgroups, but when $G$ has higher rank, for instance $G=\PSL_d(\Rb)$ and $d \geq 3$, the situation is very rigid.

\begin{theorem}[Kleiner-Leeb~\cite{KL2006}, Quint~\cite{Q2005}] Suppose $G$ is  a simple Lie group with real rank at least two and $\Gamma \leq G$ is a Zariski dense discrete subgroup. If $\Gamma$ is convex co-compact, then $\Gamma$ is a co-compact lattice in $G$. 
\end{theorem}

Although the ``symmetric space'' definition of convex co-compact subgroups leads to no interesting new examples in higher rank, Danciger-Gu{\'e}ritaud-Kassel~\cite{DGF2017} have recently introduced a different notion of convex co-compact subgroups in $G:=\PGL_d(\Rb)$ based on the action of the subgroup on the projective space $\Pb(\Rb^d)$. 

Their definition of convex co-compact subgroups requires some preliminary definitions. When $\Omega \subset \Pb(\Rb^d)$ is a properly convex domain, the \emph{automorphism group of $\Omega$} is defined to be
\begin{align*}
\Aut(\Omega) : = \left\{ g \in \PGL_d(\Rb) : g \Omega = \Omega\right\}.
\end{align*}
For a subgroup $\Lambda \leq \Aut(\Omega)$, the  \emph{full orbital limit set of $\Lambda$ in $\Omega$}, denoted by $\Lc_{\Omega}(\Lambda)$, is the set of all $x \in \partial \Omega$ where there exist  $p \in \Omega$ and a sequence $\gamma_n \in \Lambda$ such that $\gamma_n(p) \rightarrow x$. Next, let $\Cc_\Omega(\Lambda)$ denote the convex hull of $\Lc_\Omega(\Lambda)$ in $\Omega$. 

\begin{definition}\cite[Definition 1.10]{DGF2017} Suppose $\Omega \subset \Pb(\Rb^d)$ is a properly convex domain. An infinite discrete subgroup $\Lambda \leq \Aut(\Omega)$ is called \emph{convex co-compact} if $\Cc_\Omega(\Lambda)$ is non-empty and $\Lambda$ acts co-compactly on $\Cc_\Omega(\Lambda)$. 
\end{definition}

When $\Lambda$ is word hyperbolic there is a close connection between this class of discrete groups in $\PGL_d(\Rb)$ and Anosov representations, see~\cite{DGF2017} for details and~\cite{DGF2018,Z2017} for related results. The case when $\Lambda$ is not word-hyperbolic is less understood.

In this paper we study Abelian subgroups of convex co-compact groups. We also consider a more general class of convex co-compact groups defined as follows. 

\begin{definition}\label{defn:cc_naive}  Suppose $\Omega \subset \Pb(\Rb^d)$ is a properly convex domain. An infinite discrete subgroup $\Lambda \leq \Aut(\Omega)$ is called \emph{naive convex co-compact} if there exists a non-empty closed convex subset $\Cc \subset \Omega$ such that 
\begin{enumerate}
\item $\Cc$ is $\Lambda$-invariant, that is $g\Cc = \Cc$ for all $g \in \Lambda$, and
\item the quotient $\Lambda \backslash \Cc$ is compact.
\end{enumerate}
In this case, we say that $(\Omega, \Cc, \Lambda)$ is a \emph{naive convex co-compact triple}. 
\end{definition}

Clearly, if $\Lambda \leq \Aut(\Omega)$ is convex co-compact, then it is also naive convex co-compact. Further, it is straightforward to construct examples where $\Lambda \leq \Aut(\Omega)$ is naive convex co-compact, but not convex co-compact (see Example~\ref{ex:non_convex_cocompact}). In these cases, the convex subset $\Cc$ in Definition~\ref{defn:cc_naive} is a strict subset of $\Cc_\Omega(\Lambda)$.

A properly convex domain $\Omega \subset \Pb(\Rb^d)$ has a natural proper geodesic metric $H_\Omega$ called the Hilbert distance (defined in Section~\ref{sec:Hilbert_metric}). Geodesic balls in the metric space $(\Omega, H_\Omega)$ are themselves convex subsets of $\Pb(\Rb^d)$, but this metric space is $\CAT(0)$ if and only if it is isometric to real hyperbolic $(d-1)$-space (in which case $\Omega$ coincides with the unit ball in some affine chart)~\cite{KS1958}.

Despite the lack of global non-positive curvature, in this paper we establish an analogue of the well known flat torus theorem for $\CAT(0)$ groups established by Gromoll-Wolf~\cite{GW1971} and Lawson-Yau~\cite{LY1972}. In the setting of properly convex domains and the Hilbert metric, the natural analogue of totally geodesic flats are properly embedded simplices which are defined as follows. 

\begin{definition} A subset $S \subset \Pb(\Rb^d)$ is a \emph{simplex} if there exists $g \in \PGL_d(\Rb)$ and $1 \leq k \leq d$ such that 
\begin{align*}
gS = \left\{ [x_1:\dots:x_{k}:0:\dots:0] \in \Pb(\Rb^d): x_1>0,\dots, x_k>0  \right\}.
\end{align*}  
In this case we define $\dim(S)=k-1$ (notice that $S$ is diffeomorphic to $\Rb^{k-1}$) and say that the $k$ points 
\begin{align*}
g^{-1}[1:0:\dots:0], g^{-1}[0:1:0:\dots:0], \dots, g^{-1}[0:\dots:0:1:0:\dots:0] \in \partial S
\end{align*}
are the vertices of $S$. 
\end{definition}

\begin{definition} Suppose $A \subset B \subset \Pb(\Rb^d)$. Then $A$ is \emph{properly embedded in $B$} if the inclusion map $A \hookrightarrow B$ is a proper map (relative to the subspace topology). 
\end{definition} 

The main result of the paper is the following. 

\begin{theorem}\label{thm:max_abelian}(see Section~\ref{sec:pf_of_thm_max_abelian}) Suppose that $(\Omega, \Cc, \Lambda)$ is a naive convex co-compact triple. If $A \leq \Lambda$ is a maximal Abelian subgroup of $\Lambda$, then there exists a properly embedded simplex $S \subset \Cc$ such that 
\begin{enumerate}
\item $S$ is $A$-invariant, 
\item $A$ acts co-compactly on $S$, and
\item $A$ fixes each vertex of $S$. 
\end{enumerate}
Moreover, $A$ has a finite index subgroup isomorphic to $\Zb^{\dim(S)}$. 
\end{theorem}

\begin{remark} \ \begin{enumerate}
\item If $\dim S = 0$, then $S \subset \Cc$ is a fixed point of $A$ in $\Cc$ and $A$ is a finite group. If $\dim S = 1$, then $S \subset \Cc$ can be parametrized to be a unit speed geodesic in the Hilbert metric on $\Omega$. 
\item The maximality assumption in Theorem~\ref{thm:max_abelian} is necessary.   Already in the case when $\Omega$ is a simplex, there are examples of non-maximal Abelian subgroups which do not act co-compactly on any convex subset of $\Omega$ (see Example~\ref{ex:non_maximal}). 
\end{enumerate}
\end{remark}

A properly convex domain $\Omega \subset \Pb(\Rb^d)$ is called \emph{divisible} when there exists a discrete group $\Lambda \leq \Aut(\Omega)$ which acts co-compactly on all of $\Omega$. Divisible domains have been extensively studied (see the survey papers~\cite{B2008,Q2010,L2014}), but even in this very special case Theorem~\ref{thm:max_abelian} is new. 

We also note that there are a number of examples of naive convex co-compact groups which contain infinite Abelian subgroups which are \textbf{not} virtually isomorphic to $\Zb$, see for instance: \cite[Section 4]{B2006},  \cite[Theorem A]{CLM2020}, and \cite{BDL2018}.

A key step in the proof of Theorem~\ref{thm:max_abelian} is showing that the centralizer of an Abelian subgroup of a naive convex co-compact group is also a naive convex co-compact group. To state the precise result we need some terminology.

\begin{definition} Suppose that $\Omega \subset \Pb(\Rb^d)$ is a properly convex domain and $g \in \Aut(\Omega)$. Define the \emph{minimal translation length of $g$} to be
\begin{align*}
\tau_\Omega(g): = \inf_{x \in \Omega} H_\Omega(x, g x)
\end{align*}
and the \emph{minimal translation set of $g$} to be
\begin{align*}
\Min(g) = \left\{ x \in\Omega: H_\Omega(x,gx) = \tau_\Omega(g) \right\}.
\end{align*}
\end{definition}

Cooper-Long-Tillmann~\cite{CLT2015} showed that the minimal translation length of an element can be determined from its eigenvalues. In particular, given $h \in \GL_d(\Rb)$ let 
\begin{align*}
\lambda_1(h) \geq \lambda_2(h) \geq \dots \geq \lambda_d(h)
\end{align*}
denote the absolute values of the eigenvalues of $h$. Then given $g \in \PGL_d(\Rb)$ and $1 \leq i, j\leq d$ define
\begin{align*}
\frac{\lambda_i}{\lambda_j}(g) = \frac{\lambda_i(\overline{g})}{\lambda_j(\overline{g})}
\end{align*}
where $\overline{g} \in \GL_d(\Rb)$ is any lift of $g$. Then we have the following. 

\begin{proposition}\cite[Proposition 2.1]{CLT2015}\label{prop:min_trans_compute} If $\Omega \subset \Pb(\Rb^d)$ is a properly convex domain and $g \in \Aut(\Omega)$, then 
\begin{align*}
\tau_\Omega(g) = \frac{1}{2} \log \frac{\lambda_1}{\lambda_d}(g).
\end{align*}
\end{proposition}

Next, given a group $G$ and an element $g \in G$, let $C_G(g)$ denote the centralizer of $g$ in $G$. Then given a subset $X \subset G$, define
\begin{align*}
C_G(X)= \cap_{x \in X} C_G(x).
\end{align*}

We will prove the following result about centralizers and minimal translation sets of Abelian subgroups.

\begin{theorem}(see Section~\ref{sec:minimal_sets})\label{thm:centralizers_act_cocpctly} Suppose that $(\Omega, \Cc, \Lambda)$ is a naive convex co-compact triple and $A \leq \Lambda$ is an Abelian subgroup. Then
\begin{align*}
\Min_{\Cc}(A): = \Cc \cap \bigcap_{a \in A} \Min(a) 
\end{align*}
is non-empty and $C_{\Lambda}(A)$ acts co-compactly on the convex hull of $\Min_{\Cc}(A)$ in $\Omega$. 
\end{theorem}

\subsection{Outline of the paper} Sections~\ref{sec:notations} through~\ref{sec:faces} are mostly expository in nature. In Section~\ref{sec:notations} we set some basic notations, in Section~\ref{sec:examples} we construct some examples, in Section~\ref{sec:Hilbert_metric} we recall the definition of the Hilbert metric, and in Section~\ref{sec:faces} we establish some results about the faces of convex domains. 

Sections~\ref{sec:abelian_cc_action}  through~\ref{sec:pf_of_thm_max_abelian} are devoted to the proof of Theorem~\ref{thm:max_abelian}. In Section~\ref{sec:abelian_cc_action}, we give a characterization of naive convex co-compact actions of Abelian groups. In Section~\ref{sec:minimal_sets} we prove Theorem~\ref{thm:centralizers_act_cocpctly}. Finally, in Section~\ref{sec:pf_of_thm_max_abelian} we combine the results in the previous two sections to prove Theorem~\ref{thm:max_abelian}.

\subsection*{Acknowledgements} The authors thank Harrison Bray, Ludo Marquis, and Ralf Spatzier for helpful conversations. They also thank the referee for their useful comments and corrections. A. Zimmer thanks the University of Michigan for hospitality during a visit where work on this project started. 

M. Islam is partially supported by the National Science Foundation under grant 1607260 and A. Zimmer is partially supported by the National Science Foundation under grant 1904099. 

\section{Some notations}\label{sec:notations} 

In this section we set some notations that we will use for the rest of the paper. 

If $V \subset \Rb^d$ is a linear subspace, we will let $\Pb(V) \subset \Pb(\Rb^d)$ denote its projectivization. In most other cases, we will use $[o]$ to denote the projective equivalence class of an object $o$, for instance: 
\begin{enumerate}
\item if $v \in \Rb^{d} \setminus \{0\}$, then $[v]$ denotes the image of $v$ in $\Pb(\Rb^{d})$, 
\item if $\phi \in \GL_{d}(\Rb)$, then $[\phi]$ denotes the image of $\phi$ in $\PGL_{d}(\Rb)$, and 
\item if $T \in \End(\Rb^{d}) \setminus\{0\}$, then $[T]$ denotes the image of $T$ in $\Pb(\End(\Rb^{d}))$. 
\end{enumerate}
We also identify $\Pb(\Rb^d) = \Gr_1(\Rb^d)$, so for instance: if $x \in \Pb(\Rb^d)$ and $V \subset \Rb^d$ is a linear subspace, then $x \in \Pb(V)$ if and only if $x \subset V$. 

A \emph{line segment} in $\Pb(\Rb^{d})$ is a connected subset of a projective line. Given two points $x,y \in \Pb(\Rb^{d})$ there is no canonical line segment with endpoints $x$ and $y$, but we will use the following convention: if $C \subset \Pb(\Rb^d)$ is a properly convex set and $x,y \in \overline{C}$, then (when the context is clear) we will let $[x,y]$ denote the closed line segment joining $x$ to $y$ which is contained in $\overline{C}$. In this case, we will also let $(x,y)=[x,y]\setminus\{x,y\}$, $[x,y)=[x,y]\setminus\{y\}$, and $(x,y]=[x,y]\setminus\{x\}$.

Along similar lines, given a properly convex subset $C \subset \Pb(\Rb^d)$ and a subset $X \subset \overline{C}$ we will let 
\begin{align*}
{\rm ConvHull}_C(X)
\end{align*}
 denote the smallest convex subset of $\overline{C}$ which contains $X$. For instance, with our notation  $[x,y] = {\rm ConvHull}_{C}(\{x,y\})$ when $x,y \in \overline{C}$. 

Given a group $G \leq \PGL_d(\Rb)$ and a subset $X \subset \Pb(\Rb^d)$ the \emph{stabilizer of $X$ in $G$} is
\begin{align*}
\Stab_G(X) := \{ g \in G : g X = X\}.
\end{align*}
In the case when $\Omega \subset \Pb(\Rb^d)$ is a properly convex domain and $G =\Aut(\Omega)$, we will use the notation
\begin{align*}
\Stab_{\Omega}(X) := \Stab_{\Aut(\Omega)}(X). 
\end{align*}

\section{Some examples}\label{sec:examples}

In this section we construct some examples. In our first example we recall some basic properties of simplices. 

\begin{example}\label{ex:basic_properties_of_simplices} Let 
\begin{align*}
S = \left\{ [x_1:\dots:x_{d+1}] \in \Pb(\Rb^{d+1}) : x_1>0, \dots, x_{d+1}> 0\right\}.
\end{align*}
Then $S$ is a $d$-dimensional simplex. Let $G \leq \GL_{d+1}(\Rb)$ denote the group generated by the group of diagonal matrices with positive entries and the group of permutation matrices. Then 
\begin{align*}
\Aut(S) = \left\{ [g] \in \PGL_{d+1}(\Rb) : g \in G\right\}.
\end{align*}
The Hilbert metric on $S$ can be explicitly computed as:
\begin{align*}
H_S\Big([x_1:\dots:x_{d+1}], [y_1:\dots:y_{d+1}] \Big) =\max_{1\leq i,j \leq d+1} \frac{1}{2} \abs{\log \frac{x_i y_j}{y_i x_j}}.
\end{align*}
In particular, if 
\begin{align*}
\Phi\Big([x_1:\dots:x_{d+1}]\Big) = \left( \log \frac{x_2}{x_1}, \dots,  \log \frac{x_{d+1}}{x_1} \right)
\end{align*}
and ${\rm dist}$ is the distance on $\Rb^d$ given by 
\begin{align*}
{\rm dist}(v,w) = \frac{1}{2} \max\left\{ \max_{1 \leq i \leq d} \abs{v_i-w_i}, \max_{1\leq i,j \leq d} \abs{(v_i-v_j)-(w_i-w_j)} \right\},
\end{align*}
then $\Phi$ induces an isometry $(S,H_S) \rightarrow (\Rb^d, {\rm dist})$. Hence, $(S,H_S)$ is quasi-isometric to real Euclidean $d$-space. For more details, see~\cite[Proposition 1.7]{N1988}, ~\cite{dlH1993} or ~\cite{V2014}.
\end{example} 

The next example shows that the maximality assumption in Theorem~\ref{thm:max_abelian} is necessary. 

\begin{example}\label{ex:non_maximal} Again let
\begin{align*}
S = \left\{ [x_1:\dots:x_{d+1}] \in \Pb(\Rb^{d+1}) : x_1>0, \dots, x_{d+1}> 0\right\}.
\end{align*}
Then the discrete group
\begin{align*}
\Lambda := \left\{ \begin{bmatrix} e^{z_1} & & \\ & \ddots &  \\ & & e^{z_{d+1}} \end{bmatrix} : z_1,\dots,z_{d+1} \in \Zb \right\} \leq \Aut(S)
\end{align*}
acts co-compactly on $S$ and hence $(S,S,\Lambda)$ is a naive convex co-compact triple. 

Fix $1 \leq k \leq d$ and homomorphisms $\phi_1, \dots, \phi_{d+1} : \Zb^k \rightarrow \Zb$ such that 
\begin{align*}
w \in \Zb^k \rightarrow (\phi_1(w),\dots, \phi_{d+1}(w))\in \Zb^{d+1}
\end{align*}
is injective and $\phi_i \neq \phi_j$ when $i \neq j$. Then the subgroup 
\begin{align*}
A := \left\{ \begin{bmatrix} e^{\phi_1(w)} & & \\ & \ddots &  \\ & & e^{\phi_{d+1}(w)} \end{bmatrix} : w \in \Zb^k \right\}.
\end{align*}
does not act co-compactly on any convex subset on $S$. If it did, then Theorem~\ref{thm:abelian_cc_actions} implies that there exists a properly embedded simplex $S_1 \subset S$ where $A \leq \Stab_{\Lambda}(S_1)$, $A$ fixes the vertices of $S_1$, and $A$ acts co-compactly on $S_1$. But, since $\phi_i \neq \phi_j$ when $i \neq j$, the only fixed points of $A$ in $\overline{S}$ are the vertices of $S$. So the vertices of $S_1$ are also vertices of $S$. But then, since $S_1 \subset S$, we must have $S_1 = S$. Finally since $A \leq \Lambda$ has infinite index,  the quotient $A \backslash S_1 = A \backslash S$ is non-compact. So we have a contradiction.   \end{example}

The next example is a naive convex co-compact subgroup which is not convex co-compact. 

\begin{example}\label{ex:non_convex_cocompact} Suppose $\Omega \subset \Pb(\Rb^d)$ is a properly convex domain and $\Lambda \leq \Aut(\Omega)$ is a discrete group which acts co-compactly on $\Omega$. 

Let $\pi: \Rb^d \rightarrow \Pb(\Rb^d)$ be the natural projection. Then $\pi^{-1}(\Omega) = C \cup -C$ where $C \subset \Rb^d$ is some properly convex cone.   Then define
\begin{align*}
\Omega_{\star} & := \{ [(v,w)] : v,w \in C\} \subset \Pb(\Rb^{2d}), \\
\Cc_{\star}& :=\{ [(v,v)] : v \in C\}  \subset \Pb(\Rb^{2d}), \text{ and} \\
\Lambda_{\star}&:=\{ [g \oplus g] : g \in \GL_d(\Rb), [g] \in \Lambda\} \subset \PGL_{2d}(\Rb). 
\end{align*}

Then $(\Omega_{\star}, \Cc_{\star}, \Lambda_{\star})$ is a naive convex co-compact triple. We will now show that  
\begin{align*}
\Cc_{\Omega_{\star}}(\Lambda_{\star}) =\CH_{\Omega_\star} \{\Lc_{\Omega_\star}(\Lambda_\star)\} = \Omega_{\star}
\end{align*}
and hence $\Lambda_{\star} \leq \Aut(\Omega_{\star})$ is not a convex co-compact subgroup. Since $\Lambda$ acts co-compactly on $\Omega$, for every $[v] \in \partial \Omega$ there exist $p \in C$ and $g_n \in \Lambda$ such that $[v]=\lim_{n \to \infty} [g_n] [p]$ (see for instance Proposition~\ref{prop:projections_onto_faces} below).  Then, for all $t>0$, 
\[ [(v,tv)]= \lim_{n \to \infty} [g_n \oplus g_n] ~ [(p,tp)] \in \Lc_{\Omega_{\star}}(\Lambda_\star).\] Thus $\{ [(v,0)]: [v] \in \partial \Omega \} \subset \Lc_{\Omega_\star}(\Lambda_\star)$ which implies that $\{ [(v,0)]: v \in C\} \subset \Cc_{\Omega_\star}(\Lambda_\star)$. By symmetry, $\{ [(0,w)]: w \in C\} \subset \Cc_{\Omega_\star}(\Lambda_\star)$. Thus $ \Cc_{\Omega_\star}(\Lambda_\star) =\Omega_\star$. 

We can also ``thicken'' $\Cc_{\star}$ to obtain other naive convex co-compact triples that do not correspond to convex co-compact groups. By Proposition \ref{prop:convexity-of-nbd}, 
\begin{align*}
\Cc_{R,\star} :=  \left\{  y \in \Omega_{\star} : H_{\Omega_{\star}}(y, \Cc_{\star}) \leq R \right\}
\end{align*}
is a closed convex subset of $\Omega_\star$. Thus $(\Omega_{\star},\Cc_{R,\star} , \Lambda_{\star})$ is also a naive convex co-compact triple. 
\end{example}

\section{Convexity and the Hilbert metric}\label{sec:Hilbert_metric}  

In this section we recall the definition of convex sets in projective space and the classical Hilbert metric on properly convex (relatively) open sets. 

\begin{definition} \ 
\begin{enumerate}
\item A subset $C \subset \Pb(\Rb^d)$ is \emph{convex} if there exists an affine chart $\mathbb{A}$ of $\Pb(\Rb^d)$ where $C \subset \mathbb{A}$ is a convex subset. 
\item A subset $C \subset \Pb(\Rb^d)$ is \emph{properly convex} if there exists an affine chart $\mathbb{A}$ of $\Pb(\Rb^d)$ where $C \subset \mathbb{A}$ is a bounded convex subset. 
\item When $C$ is a properly convex set which is open in $\Pb(\Rb^d)$ we say that $C$ is a \emph{properly convex domain}.
\end{enumerate}
\end{definition}

Notice that if $C \subset \Pb(\Rb^d)$ is convex, then $C$ is a convex subset of every affine chart that contains it. We also make the following topological definitions.

\begin{definition}\label{defn:topology} Suppose $C \subset \Pb(\Rb^d)$ is a properly convex set. The \emph{relative interior of $C$}, denoted by $\relint(C)$, is  the interior of $C$ in $\Pb( \Spanset C)$. In the case that $C = \relint(C)$, then $C$ is said to be \emph{open in its span}. The \emph{boundary of $C$} is $\partial C : = \overline{C} \setminus \relint(C)$, the \emph{ideal boundary of $C$} is
\begin{align*}
\partiali C := \partial C \setminus C,
\end{align*}
and the \emph{non-ideal boundary of $C$} is
\begin{align*}
\partialni C := \partial C \cap C
\end{align*}
Finally, we define $\dim C$ to be the dimension of $\relint(C)$ (notice that $\relint(C)$ is homeomorphic to $\Rb^{\dim C}$). 
\end{definition}

Recall that a subset $A \subset B \subset \Pb(\Rb^d)$ is properly embedded if the inclusion map $A \hookrightarrow B$ is proper. With the notation in Definition~\ref{defn:topology} we have the following characterization of properly embedded subsets. 

\begin{observation} Suppose $C \subset \Pb(\Rb^d)$ is a properly convex set. A convex subset $S \subset C$ is properly embedded if and only if $\partiali S \subset \partiali C$. 
\end{observation}

For distinct points $x,y \in \Pb(\Rb^{d})$ let $\overline{xy}$ be the projective line containing them. Suppose $C \subset \Pb(\Rb^{d})$ is a properly convex set which is open in its span. If $x,y \in C$ are distinct let $a,b$ be the two points in $\overline{xy} \cap \partial C$ ordered $a, x, y, b$ along $\overline{xy}$. Then define \emph{the Hilbert distance between $x$ and $y$} to be
\begin{align*}
H_{C}(x,y) = \frac{1}{2}\log [a, x,y, b]
\end{align*}
 where 
 \begin{align*}
 [a,x,y,b] = \frac{\abs{x-b}\abs{y-a}}{\abs{x-a}\abs{y-b}}
 \end{align*}
 is the cross ratio. Using the invariance of the cross ratio under projective maps and the convexity of $C$ it is possible to establish the following (see for instance~\cite[Section 28]{BK1953}). 
 
 \begin{proposition}\label{prop:hilbert_basic}
Suppose $C \subset \Pb(\Rb^{d})$ is a properly convex set which is open in its span. Then $H_{C}$ is a complete $\Aut(C)$-invariant proper metric on $C$ which generates the standard topology on $C$. Moreover, if $p,q \in C$, then there exists a geodesic joining $p$ and $q$ whose image is the line segment $[p,q]$.
\end{proposition}

Convexity is preserved under taking $r$-neighbourhoods in the Hilbert metric of closed convex sets.

\begin{proposition}
\label{prop:convexity-of-nbd}\cite[Result 18.9]{HB1955}
If $~\Omega$ is a properly convex domain, $\Dc \subset \Omega$ is a non-empty closed convex set, and $r\geq 0$, then \[ \Nc_{r}(\Dc):=\{ x \in \Omega : \hil(x, \Dc) < r \}\] is a convex subset of $\Omega$. 
\end{proposition}

\begin{remark}
A proof can also be found in~\cite[Corollary 1.10]{CLT2015}.
\end{remark}

Using an argument of Frankel~\cite{Fra1989} we define a notion of ``center of mass'' for a compact set in a properly convex domain. Let $\Kc_d$ denote the set of all pairs $(\Omega, K)$ where $\Omega \subset \Pb(\Rb^d)$ is a properly convex domain and $K \subset \Omega$ is a compact subset. 

\begin{proposition}\label{prop:center_of_mass} There exists a function
\begin{align*}
(\Omega, K) \in \Kc_d \, \longmapsto \, {\rm CoM}_\Omega(K) \in \Pb(\Rb^d)
\end{align*}
such that:
\begin{enumerate}
\item ${\rm CoM}_\Omega(K)  \in {\rm ConvHull}_\Omega(K)$, 
\item ${\rm CoM}_\Omega(K) = {\rm CoM}_\Omega({\rm ConvHull}_\Omega(K))$, and
\item if $g \in \PGL_d(\Rb)$, then $g{\rm CoM}_\Omega(K)={\rm CoM}_{g\Omega}(gK)$,
\end{enumerate}
for every $(\Omega, K) \in \Kc_d$. 
\end{proposition} 

The following argument is due to Frankel~\cite[Section 12]{Fra1989} who constructed a ``holomorphic center of mass'' associated to a compact subset of a bounded convex domain in $\Cb^d$. Frankel's construction used the Kobayashi metric instead of the Hilbert metric and is equivariant under biholomorphisms instead of real projective transformations. An alternative approach to constructing a projective ``center of mass'' is given in~\cite[Lemma 4.2]{L2014}. 

\begin{proof} Fix some $(\Omega, K) \in \Kc_d$. We define a sequence of convex sets $C_0 \supset C_1 \supset  C_2 \dots$ as follows. First let 
\begin{align*}
C_0 =  {\rm ConvHull}_\Omega(K).
\end{align*}
Then supposing that $C_0, \dots, C_n$ have been selected, define
\begin{align*}
C_n(r) =C_n \cap  \bigcap_{c \in C_n} \{ p \in \Omega : H_\Omega(p, c) \leq r \}
\end{align*}
and
\begin{align*}
r_n = \min \{ r > 0 : C_n(r) \neq \emptyset\}.
\end{align*}
Then define $C_{n+1}:=C_n(r_n)$. Then $C_{n+1}$ is closed, convex, and $C_{n+1} \subset C_n$. Moreover, if $\dim C_n \geq 1$, then $\dim C_{n+1} < \dim C_n$ (otherwise $r_n$ was not minimal). So
\begin{align*}
{\rm CoM}_\Omega(K) : = C_{d}
\end{align*}
is a point in $\Omega$. It is clear from the construction that this definition satisfies conditions (1), (2), and (3). 
\end{proof}

\section{The faces of a convex domain}\label{sec:faces}

Given a properly convex domain $\Omega \subset \Pb(\Rb^d)$ and $x \in \overline{\Omega}$ let $F_\Omega(x)$ denote the \emph{open face} of $x$, that is 
\begin{align*}
F_\Omega(x) = \{ x\} \cup \left\{ y \in \overline{\Omega} : \text{ $\exists$ an open line segment in $\overline{\Omega}$ containing $x$ and $y$} \right\}.
\end{align*}
Notice that $F_\Omega(x) = \Omega$ when $x \in \Omega$. 

\begin{observation}\label{obs:faces} Suppose $\Omega \subset \Pb(\Rb^d)$ is a properly convex domain. 
\begin{enumerate}
\item $F_\Omega(x)$ is open in its span,
\item $y \in F_\Omega(x)$ if and only if $x \in F_\Omega(y)$ if and only if $F_\Omega(x) = F_\Omega(y)$,
\item if $y \in \partial F_\Omega(x)$, then $F_\Omega(y) \subset \partial F_\Omega(x)$,
\item if $x, y \in \overline{\Omega}$, $z \in (x,y)$, $p \in F_{\Omega}(x)$, and $q \in F_{\Omega}(y)$, then 
\begin{align*}
(p,q) \subset F_\Omega(z).
\end{align*}
In particular, $(p,q) \subset \Omega$ if and only if $(x,y) \subset \Omega$.
\end{enumerate}
\end{observation}

\begin{proof} These are all simple consequences of convexity. \end{proof}

\subsection{The Hilbert metric and faces} We now observe several results which relate the faces of a convex domain with the Hilbert metric. 

\begin{proposition}\label{prop:dist_est_and_faces} Suppose $\Omega \subset \Pb(\Rb^d)$ is a properly convex domain, $x_n$ is a sequence in $\Omega$, and $x_n \rightarrow x \in \overline{\Omega}$. If $y_n$ is another sequence in $\Omega$, $y_n \rightarrow  y \in \overline{\Omega}$, and 
\begin{align*}
\liminf_{n \rightarrow \infty} H_\Omega(x_n,y_n) < + \infty,
\end{align*}
then $y \in F_\Omega(x)$ and 
\begin{align*}
H_{F_\Omega(x)}(x,y) \leq \liminf_{n \rightarrow \infty} H_\Omega(x_n,y_n).
\end{align*}
\end{proposition}

\begin{proof} This is a straightforward consequence of the definition of the Hilbert metric. \end{proof}

Given a properly convex domain $\Omega \subset \Pb(\Rb^d)$, let $H_\Omega^{\Haus}$ denote the \emph{Hausdorff distance} on subsets of $\Omega$ induced by $H_\Omega$, that is: for subsets $A,B \subset \Omega$ define
\begin{align*}
H_\Omega^{\Haus}(A,B) = \max \left\{ \sup_{a \in A}\inf_{b \in B} H_\Omega(a,b), \, \sup_{b \in B}\inf_{a \in A} H_\Omega(a,b) \right\}.
\end{align*}

\begin{proposition}\label{prop:Crampons_dist_est_2} Suppose $\Omega \subset \Pb(\Rb^d)$ is a properly convex domain. Assume $p_1,p_2,q_1,q_2 \in \overline{\Omega}$, $F_\Omega(p_1) = F_\Omega(p_2)$, and $F_\Omega(q_1) = F_\Omega(q_2)$. If $(p_1,q_1) \cap \Omega \neq \emptyset$, then
\begin{align*}
H_{\Omega}^{\Haus}\Big((p_1,q_1), (p_2, q_2) \Big) \leq \max\{ H_{F_\Omega(p_1)}(p_1,p_2), H_{F_\Omega(q_1)}(q_1,q_2) \}.
\end{align*}
\end{proposition}

\begin{remark} Since $(p_1,q_1) \cap \Omega \neq \emptyset$, Observation~\ref{obs:faces} part (4)  implies that 
\begin{align*}
(p_1,q_1),(p_2,q_2) \subset \Omega. 
\end{align*}
\end{remark}

\begin{proof} Set $R:=\max\{ H_{F_\Omega(p_1)}(p_1,p_2), H_{F_\Omega(q_1)}(q_1,q_2) \}.$ Let $p_{2,n},q_{2,n} \in (p_2,q_2)$ be sequences such that $p_2=\lim_{n \to \infty} p_{2,n}$ and  $q_2=\lim_{n \to \infty} q_{2,n}$. Then there exists $R_n \rightarrow R$ such that 
\begin{align*}
p_{2,n},q_{2,n} \in \Nc_{R_n}((p_1,q_1)).
\end{align*}
Then Proposition \ref{prop:convexity-of-nbd} implies that $[p_{2,n}, q_{2,n}] \subset \Nc_{R_n}((p_1,q_1))$. Thus $(p_2,q_2) \subset \overline{\Nc_{R}((p_1,q_1))}$. By symmetry, $(p_1,q_1) \subset \overline{\Nc_{R}((p_2,q_2))}$. 
 \end{proof}

We will also use the following estimate. 

\begin{lemma}[{Crampon~\cite[Lemma 8.3]{C2009}}]\label{lem:Crampons_dist_est} Suppose that $\sigma_1, \sigma_2 : [0,T] \rightarrow \Omega$ are two unit speed projective line geodesics, then 
\begin{align*}
H_\Omega(\sigma_1(t), \sigma_2(t)) \leq H_\Omega(\sigma_1(0), \sigma_2(0))+H_\Omega(\sigma_1(T), \sigma_2(T))
\end{align*}
for $0 \leq t \leq T$. 
\end{lemma}

\subsection{Dynamics of automorphisms} The next two results relate the faces of a convex domain with the behavior of automorphisms. 

In the next result we view $\PGL_d(\Rb)$ as a subset of $\Pb(\End(\Rb^d))$. 

\begin{proposition}\label{prop:dynamics_of_automorphisms}
Suppose $\Omega \subset \Pb(\Rb^d)$ is a properly convex domain, $p_0 \in \Omega$, and $g_n \in \Aut(\Omega)$ is a sequence such that 
\begin{enumerate}
\item $g_n (p_0) \rightarrow x \in \partial \Omega$, 
\item $g_n^{-1} (p_0) \rightarrow y \in \partial \Omega$, and
\item $g_n$ converges in $\Pb(\End(\Rb^d))$ to $T \in \Pb(\End(\Rb^d))$. 
\end{enumerate}
Then 
${\rm image}(T) \subset \Spanset F_\Omega(x)$, $\Pb(\ker T) \cap \Omega = \emptyset$, and $y \in \Pb(\ker T)$. 
\end{proposition} 

\begin{proof} For $v \in \Rb^d$ let $\norm{v}$ be the standard Euclidean norm of $v$ and for $S \in \End(\Rb^d)$ let $\norm{S}$ denote the associated operator norm. Also let $e_1,\dots, e_d$ denote the standard basis of $\Rb^d$. 

Notice that 
\begin{align*}
T(p) = \lim_{n \rightarrow \infty} g_n(p)
\end{align*}
for all $p \notin \Pb(\ker T)$. 

We can pick a lift $\overline{g}_n \in \GL_d(\Rb)$ of each $g_n$ with $\norm{\overline{g}_n}=1$ such that $\overline{g}_n \rightarrow \overline{T}$ in $\End(\Rb^d)$ and $\overline{T}$ is a lift of $T$. 

\medskip

\noindent \textbf{Claim 1:} $\Pb(\ker T) \cap \Omega = \emptyset$. 

\medskip

\noindent \emph{Proof of Claim 1:} Using the singular value decomposition, we can find $k_{n,1}, k_{n,2} \in {\rm O}(d)$ and $1=\sigma_{1,n} \geq \dots \geq \sigma_{d,n}>0$ such that 
\begin{align*}
\overline{g}_n = k_{n,1} \begin{pmatrix} \sigma_{1,n} & &  \\ & \ddots & \\ & & \sigma_{d,n} \end{pmatrix} k_{n,2}.
\end{align*}
By passing to a subsequence we can suppose that $k_{n,1} \rightarrow k_1$, $k_{n,2} \rightarrow k_2$, and 
\begin{align*}
\chi_j : = \lim_{n \rightarrow \infty} \sigma_{j,n} \in [0,1]
\end{align*}
exists for every $1 \leq j \leq d$. Then 
\begin{align*}
\overline{T} = k_1  \begin{pmatrix} 1 & &  &  \\ & \chi_2 & & \\ & & \ddots & \\ & & & \chi_d \end{pmatrix} k_2.
\end{align*}
Let 
\begin{align}
\label{eq:defn_of_m}
m := \max\left\{ j : \chi_j >0\right\}.
\end{align}
Then $\ker T = k_2^{-1} \Spanset\{ e_{m+1}, \dots, e_d\}$. 

Suppose for a contradiction that there exists $[v] \in \Pb(\ker T) \cap \Omega$. Let 
\begin{align*}
v_n := k_{n,2}^{-1} k_2v \in k_{n,2}^{-1} \Spanset\{ e_{m+1}, \dots, e_d\}.
\end{align*}
Since $\Omega$ is open and $v_n \rightarrow v$, by passing to a tail we can assume that there exists some $\epsilon > 0$ such that 
\begin{align*}
\Big\{ \left[v_n+sk_{n,2}^{-1}e_1\right] : \abs{s} <\epsilon\Big\}\subset \Omega
\end{align*}
for all $n \geq 0$. By passing to a subsequence we can suppose that 
\begin{align*}
w:=\lim_{n \rightarrow \infty} \frac{1}{\norm{\overline{g}_n v_n}} \overline{g}_n v_n \in \Rb^d
\end{align*}
exists. Now fix $t \in \Rb$ and let $t_n :=  \norm{\overline{g}_n v_n} t$. Since $\norm{\overline{g}_n v_n} \leq \sigma_{m+1,n}\norm{v_n}$ and 
 \begin{align*}
 \lim_{n \rightarrow \infty} \sigma_{m+1,n} =0,
 \end{align*}
  for $n$ sufficiently large we have $\abs{t_n} < \epsilon$. Then 
 \begin{align*}
 \left[ w + t k_1 e_1\right] & = \lim_{n \rightarrow \infty} \left[ \frac{1}{\norm{\overline{g}_n v_n}} \left( \overline{g}_n v_n + t_n k_{n,1} e_1 \right) \right] \\
 & = \lim_{n \rightarrow \infty} \left[ \frac{1}{\norm{\overline{g}_n v_n}} \left( \overline{g}_n v_n + t_n  \overline{g}_n k_{n,2}^{-1} e_1 \right) \right] \\
 &  = \lim_{n \rightarrow \infty} g_n \left[v_n+t_nk_{n,2}^{-1}e_1\right] \in \overline{\Omega}.
 \end{align*}
 Since $t$ is arbitrary, we see that 
 \begin{align*}
 \{ [w+t k_1 e_1] : t \in \Rb\} \subset \overline{\Omega}
 \end{align*}
 which contradicts the fact that $\Omega$ is properly convex. So $\Pb(\ker T) \cap\Omega = \emptyset$. 
 
 \medskip

\noindent \textbf{Claim 2:} $T(\Omega) \subset F_\Omega(x)$. In particular, 
\begin{align*}
{\rm image}(T) \subset \Spanset F_\Omega(x).
\end{align*}

\noindent \emph{Proof of Claim 2:} Since $\Pb(\ker T) \cap \Omega = \emptyset$, 
\begin{align*}
T(p) = \lim_{n \rightarrow \infty} g_n(p)
\end{align*}
for all $p \in \Omega$. Since $g_n(p_0) \rightarrow x$ and 
\begin{align*}
H_\Omega(g_n(p), g_n(p_0)) = H_\Omega(p,p_0),
\end{align*}
Proposition~\ref{prop:dist_est_and_faces} implies that  $T(\Omega) \subset F_\Omega(x)$. 

 \medskip

\noindent \textbf{Claim 3:}  $y \in \Pb(\ker T)$. 

\medskip

\noindent \emph{Proof of Claim 3:} Notice that 
\begin{align*}
\overline{h}_n := k_{n,2}^{-1} \begin{pmatrix} \sigma_{1,n}^{-1} & &  \\ & \ddots & \\ & & \sigma_{d,n}^{-1} \end{pmatrix} k_{n,1}^{-1}
\end{align*}
is a lift of $g_n^{-1}$. Since  $1=\sigma_{1,n} \geq \dots \geq \sigma_{d,n}>0$, we can pass to a subsequence and assume that $\sigma_{d,n} \overline{h}_n$ converges in $\End(\Rb^d)$ to some non-zero $S \in \End(\Rb^d)$. Then $g_n^{-1}$ converges in $\Pb(\End(\Rb^d))$ to $[S] \in \Pb(\End(\Rb^d))$. Claim 1 applied to $g_n^{-1}$ implies that $\Pb(\ker S) \cap \Omega = \emptyset$. So 
\begin{align*}
S(p_0) = \lim_{n \rightarrow \infty} g_n^{-1}(p_0) = y.
\end{align*}
Further, Equation~\eqref{eq:defn_of_m} implies that
\begin{align*}
{\rm image}(S) \subset k_2^{-1} \Spanset\{ e_{m+1}, \dots, e_d\} = \ker T.
\end{align*}
So $y \in \Pb(\ker T )$. 
\end{proof}

Given a group $G \leq \PGL_d(\Rb)$ define $\overline{G}^{\End}$ to be the closure of the set 
\begin{align*}
\{ g \in \GL_d(\Rb) : [g] \in G\}
\end{align*}
in $\End(\Rb^d)$. 

\begin{proposition}\label{prop:projections_onto_faces}
Suppose $\Omega \subset \Pb(\Rb^d)$ is a properly convex domain, $\Cc \subset \Omega$ is a non-empty closed convex subset, and $G \leq \Stab_\Omega(\Cc)$ acts co-compactly on $\Cc$. If $x \in \partiali \Cc$, then there exists $T \in \overline{G}^{\End}$ such that 
\begin{enumerate}
\item $\Pb(\ker T) \cap \Omega = \emptyset$,
\item $T(\Omega) = F_\Omega(x)$, and
\item $T(\Cc) = F_\Omega(x) \cap \partiali \Cc$.
\end{enumerate}
\end{proposition}

\begin{proof} Fix some $p_0 \in \Cc$ and a sequence $p_n \in [p_0, x)$ with $p_n \rightarrow x$. Since $G$ acts co-compactly on $\Cc$, there exists $R > 0$ and a sequence $g_n \in G$ such that 
\begin{align*}
H_\Omega(g_n p_0, p_n) \leq R
\end{align*}
for all $n \geq 0$. 

As before, for $S \in \End(\Rb^d)$ let $\norm{S}$ be the operator norm associated  to the standard Euclidean norm. Let $\overline{g}_n \in \GL_d(\Rb)$ be a lift of $g_n$ with $\norm{\overline{g}_n}=1$. By passing to a subsequence we can suppose that $\overline{g}_n \rightarrow T$ in $\End(\Rb^d)$. Proposition~\ref{prop:dynamics_of_automorphisms} implies that $\Pb(\ker T) \cap \Omega = \emptyset$ and $T(\Omega) \subset F_\Omega(x)$. Then 
\begin{align*}
T(p) = \lim_{n \rightarrow \infty} g_n(p)
\end{align*}
for all $p \in \Omega$.

\medskip

\noindent \textbf{Claim 1:} $T(\Omega) = F_\Omega(x)$. 

\medskip

\noindent \emph{Proof of Claim 1:} We only need to show that $F_\Omega(x) \subset T(\Omega)$. So fix $y \in F_\Omega(x)$. Then we can pick $y_n \in [p_0, y)$ such that 
\begin{align*}
\sup_{n \geq 0} H_\Omega(y_n, p_n) <\infty.
\end{align*}
Thus 
\begin{align*}
\sup_{n \geq 0} H_\Omega(g_n^{-1} y_n, p_0) <\infty.
\end{align*}
So there exists $n_j \rightarrow\infty$ so that the limit 
\begin{align*}
q:=\lim_{j \rightarrow \infty} g_{n_j}^{-1} y_{n_j}
\end{align*}
exists in $\Omega$. Then 
\begin{align*}
T(q) = \lim_{n \rightarrow \infty} g_n(q) = \lim_{j \rightarrow \infty} g_{n_j} g_{n_j}^{-1} y_{n_j} = \lim_{j \rightarrow \infty}y_{n_j} = y.
\end{align*}
Hence $F_\Omega(x) \subset T(\Omega)$. 

\medskip

\noindent \textbf{Claim 2:} $T(\Cc) = F_\Omega(x) \cap \partiali \Cc$.

\medskip

\noindent \emph{Proof of Claim 2:} This is almost identical to the proof of Claim 1.  

\end{proof}

\section{Abelian convex co-compact actions}\label{sec:abelian_cc_action}

In this section we show that every naive convex co-compact action of an Abelian group comes from ``fattening'' a properly embedded simplex. 

\begin{theorem}\label{thm:abelian_cc_actions} Suppose $\Omega \subset \Pb(\Rb^d)$ is a properly convex domain, $\Cc \subset \Omega$ is a non-empty closed convex subset, and $G \leq \Stab_\Omega (\Cc)$. If $G$ is Abelian and acts co-compactly on $\Cc$, then there exists a properly embedded simplex $S \subset \Cc$ where 
\begin{enumerate}
\item $G \leq \Stab_\Omega (S)$, 
\item $G$ acts co-compactly on $S$, and 
\item $G$ fixes each vertex of $S$. 
\end{enumerate}
\end{theorem}

\begin{remark}Notice that we do not assume that $G$ is a discrete subgroup of $\Aut(\Omega)$. 
\end{remark}

The rest of the section is devoted to the proof of the theorem. We will induct on 
\begin{align*}
\dim \Omega + \dim \Cc.
\end{align*}
The base case, when $\dim \Omega = 1$ and  $\dim \Cc=0$, is trivial. 

Suppose that $\Omega, \Cc, G$ satisfy the hypothesis of the theorem. From Proposition~\ref{prop:center_of_mass} we immediately obtain the following.

\begin{observation} If $\Cc$ is compact, then $G$ fixes the point ${ \rm CoM}_{\Omega}(\Cc)$. \end{observation} 

Since a point is a 0-dimensional simplex, the above observation completes the proof in the case when $\Cc$ is compact. So for the rest of the argument we assume that $\Cc$ is non-compact and hence $\partiali \Cc \neq \emptyset$. Our first goal will be to find a finite number of fixed points $x_1, \dots, x_k$ of $G$ in $\partiali \Cc$ such that 
\begin{align*}
 {\rm ConvHull}_\Omega \{ x_1, \dots, x_k\} \cap \Omega
\end{align*}
is non-empty. 

\begin{lemma}\label{lem:abelian_preserves_faces} If $x \in \partiali \Cc$ and $F: = F_\Omega(x)$, then 
\begin{enumerate}
\item $G \leq \Stab_\Omega(F)$, 
\item $G \leq \Stab_\Omega(F \cap \partiali\Cc)$, and
\item  $G$ acts co-compactly on $F \cap \partiali\Cc$.
\end{enumerate}
 \end{lemma}

\begin{proof} 
By Proposition~\ref{prop:projections_onto_faces} there exists some $T \in \overline{G}^{\End}$ such that $\Pb(\ker T) \cap \Omega= \emptyset$, $T(\Omega)=F$, and $T(\Cc) = F \cap \partiali\Cc$. Since $G$ is Abelian, $T \circ g = g \circ T$ for every $g \in G$. 

Then for $g \in G$ we have
\begin{align*} 
gF = gT(\Omega) = T(g \Omega)=T(\Omega) = F.
\end{align*}
Since $g \in G$ was arbitrary, $G \leq \Stab_\Omega(F)$. Then $G \leq \Stab_\Omega(F \cap \partiali\Cc)$ since $G \leq \Stab_\Omega(\Cc)$.

Since $G$ acts co-compactly on $\Cc$, there exists a compact set $K \subset \Cc$ such that $G \cdot K = \Cc$. Since $\Pb(\ker T) \cap\Omega = \emptyset$, the map 
\begin{align*}
p \in \Omega ~\mapsto~ T(p) \in F_\Omega(x)
\end{align*}
is continuous. So  $K_F : = T(K)$ is a compact subset of $F \cap \partiali \Cc$. Then 
\begin{align*}
G \cdot K_F = G \cdot T(K) = T( G \cdot K) = T(\Cc) =F \cap \partiali \Cc.
\end{align*}
So $G$ acts co-compactly on $F \cap \partiali \Cc$.
\end{proof}

\begin{lemma} There exists a properly embedded 1-dimensional simplex $\ell \subset \Cc$. \end{lemma}

\begin{proof} Fix some $x_0 \in \Cc$. Since $\Cc$ is non-compact, there exists some $x \in \partiali \Cc$. Then pick $x_n \in [x_0, x)$ converging to $x$. Since $[x_0,x) \subset \Cc$ and $G$ acts co-compactly on $\Cc$, there exist $r > 0$ and a sequence $g_n \in G$ such that 
\begin{align*}
H_\Omega(g_n x_n, x_0) \leq r
\end{align*}
for all $n \geq 0$. By passing to a subsequence we can suppose that $g_n x_n \rightarrow q \in \Cc$. By passing to another subsequence we can assume that $g_n \cdot (x_0, x)$ converges to a properly embedded 1-dimensionial simplex $\ell \subset \Cc$. 
\end{proof}

\begin{lemma} There exists a finite number of fixed points $x_1, \dots, x_m$ of $G$ in $\partiali \Cc$ such that 
\begin{align*}
{\rm ConvHull}_\Omega \{ x_1, \dots, x_m\} \cap \Omega
\end{align*}
is non-empty. 
\end{lemma}

\begin{proof} By the previous lemma there exists a properly embedded 1-dimensional simplex $\ell \subset \Cc$. Let  $y_1, y_2$ be the endpoints of $\ell$ and let $F_j:= F_\Omega(y_j)$. 

First, we will find a finite number of fixed points $a_1, \dots, a_k$ of $G$ in $\overline{F}_1 \cap \partiali \Cc$ such that 
\begin{align*}
{\rm ConvHull}_\Omega \left\{ a_1, \dots, a_k \right\} \cap F_1
\end{align*}
is non-empty. By Lemma~\ref{lem:abelian_preserves_faces} and induction there exists a properly embedded simplex $S_1 \subset F_1$ where $G$ fixes each vertex of $S_1$. Let $a_1,\dots,a_k$ be the vertices of $S_1$. Then
\begin{align*}
S_1 = {\rm ConvHull}_\Omega \left\{ a_1, \dots, a_k \right\} \cap F_1
\end{align*}
is non-empty. 

Applying the same argument to $F_2$ yields a finite number of fixed points $b_1, \dots, b_n$ of $G$ in $\overline{F_2} \cap \partiali \Cc$ such that 
\begin{align*}
{\rm ConvHull}_\Omega \left\{ b_1, \dots, b_n \right\} \cap F_2
\end{align*}
is non-empty. 

Finally, we claim that 
\begin{align*}
{\rm ConvHull}_\Omega \left\{ a_1, \dots, a_k, b_1, \dots, b_n \right\} \cap \Omega \neq \emptyset.
\end{align*}
is non-empty. By construction, this convex hull contains some $a^\prime \in F_1$ and $b^\prime \in F_2$. Since $y_1 \in F_1$, $y_2 \in F_2$, and $\ell=(y_1, y_2) \subset \Omega$, Observation~\ref{obs:faces} part (4) implies that $(a^\prime, b^\prime) \subset \Omega$. Then 
\begin{align*}
(a^\prime, b^\prime) \subset {\rm ConvHull}_\Omega\left\{ a_1, \dots, a_k, b_1, \dots, b_n \right\} \cap \Omega
\end{align*}
and we are done. 
\end{proof}

By the previous lemma, there exist fixed points $x_1, \dots, x_m$ of $G$ in $\partiali \Cc$ such that 
\begin{align*}
S:={\rm ConvHull}_\Omega \{ x_1, \dots, x_m\} \cap \Omega
\end{align*}
is non-empty. We can also assume  that $m$ is minimal in the following sense: if $y_1, \dots, y_k$ are fixed points of $G$ in $\partiali \Cc$ with $k < m$, then 
\begin{align*}
{\rm ConvHull}_\Omega \{ y_1, \dots, y_k\} \cap \Omega = \emptyset. 
\end{align*}
Also, notice that $m \geq 2$ since $x_1,\dots, x_m \in \partiali\Cc$ and $S \neq \emptyset$. We complete the proof of Theorem~\ref{thm:abelian_cc_actions} by proving the following.

\begin{lemma} $S$ is a properly embedded simplex in $\Omega$, $G \leq \Stab_\Omega(S)$, $G$ acts co-compactly on $S$, and $G$ fixes each vertex of $S$.
\end{lemma} 

\begin{proof} 
Let $d_0:=\dim S$ (in the sense of Definition~\ref{defn:topology}). We claim that $d_0=m-1$. By definition, 
\begin{align*}
d_0 = \dim \Pb( \Spanset \{ x_1,\dots, x_m\}) \leq m-1.
\end{align*}
For the reverse inequality, fix $p \in S$. Then by Carath\'eodory's convex hull theorem there exists $x_{i_1}, \dots, x_{i_k}$ with $k \leq d_0+1$ such that 
\begin{align*}
p \in {\rm ConvHull}_\Omega \{ x_{i_1},\dots, x_{i_{k}}\}.
\end{align*}
Hence
\begin{align*}
\emptyset \neq {\rm ConvHull}_\Omega \{ x_{i_1},\dots, x_{i_{k}}\} \cap \Omega.
\end{align*}
So by our minimality assumption we must have $k = m$ and so $m \leq d_0+1$. So $m=d_0+1$. Thus $x_1,\dots, x_m$ are linearly independent and hence $S$ is a simplex with vertices $\{x_1,\dots, x_m\}$.

By the minimality property, for any proper subset $\{ x_{i_1},\dots, x_{i_k}\} \subset \{x_1,\dots, x_m\}$ we have
\begin{align*}
\emptyset={\rm ConvHull}_\Omega \{ x_{i_1}, \dots, x_{i_k}\} \cap \Omega.
\end{align*}
So $S$ is a properly embedded simplex of $\Omega$.

By construction $G \leq \Stab_\Omega(S)$ and $G$ fixes each vertex of $S$. Finally, since $S \subset \Cc$ is a closed subset and $G$ acts co-compactly on $\Cc$, we see that $G$ acts co-compactly on $S$. 
\end{proof}

\section{Centralizers and minimal translation sets}\label{sec:minimal_sets}

In this section we prove Theorem~\ref{thm:centralizers_act_cocpctly} which we restate here. 

\begin{theorem}Suppose that $(\Omega, \Cc, \Lambda)$ is a naive convex co-compact triple and $A \leq \Lambda$ is an Abelian subgroup. Then
\begin{align*}
\Min_{\Cc}(A): = \Cc \cap \bigcap_{a \in A} \Min(a) 
\end{align*}
is non-empty and $C_{\Lambda}(A)$ acts co-compactly on ${ \rm ConvHull}_\Omega( \Min_{\Cc}(A))$. 
\end{theorem}

The proof the theorem will use the following observations about minimal translation sets.

\begin{observation}\label{obs:intersection_with_subspace} Suppose that $\Omega \subset \Pb(\Rb^d)$ is a properly convex domain and $g \in \Aut(\Omega)$. If $V \subset \Rb^d$ is a linear subspace where $\dim V > 1$,  $\Omega\cap\Pb(V) \neq \emptyset$, and $V$ is $g$-invariant, then 
\begin{align*}
\tau_{\Omega \cap \Pb(V)}(g) = \tau_\Omega(g).
\end{align*}
\end{observation}

\begin{proof} By the definition of the Hilbert metric $H_\Omega |_{\Pb(V)\times \Pb(V)} = H_{\Omega \cap \Pb(V)}$. Hence $\tau_\Omega(g) \leq \tau_{\Omega \cap \Pb(V)}(g)$. On the other hand, $g|_{V} \in \Aut(\Omega \cap \Pb(V))$ and so Proposition \ref{prop:min_trans_compute} implies that there exists $1\leq i<j \leq d$ such that 
\begin{align*}
\tau_{\Omega \cap \Pb(V)}(g|_V) = \dfrac{1}{2}\log\dfrac{\lambda_i}{\lambda_j}  (g).
\end{align*}
So applying Proposition \ref{prop:min_trans_compute} to $g$ yields 
\begin{equation*}
\tau_{\Omega \cap \Pb(V)}(g)= \dfrac{1}{2}\log\dfrac{\lambda_i}{\lambda_j}  (g) \leq \dfrac{1}{2}\log\dfrac{\lambda_1}{\lambda_d}  (g)=\tau_{\Omega}(g). \qedhere
\end{equation*} 
\end{proof}

\begin{proposition}\label{prop:min_set_inv_simplex} Suppose that $\Omega \subset \Pb(\Rb^d)$ is a properly convex domain and $S \subset \Omega$ is a properly embedded simplex. If $g \in \Aut(\Omega)$ fixes every vertex of $S$, then $S \subset \Min(g)$. 
\end{proposition}

\begin{proof} If $\dim S = 0$, then $S$ is a fixed point of $g$ and hence $S \subset \Min(g)$. So suppose that $\dim S \geq 1$. Then $S = \Omega \cap \Pb(\Spanset S)$ and using Observation~\ref{obs:intersection_with_subspace} there is no loss of generality in assuming that $S = \Omega$. Then the Proposition follows from Example~\ref{ex:basic_properties_of_simplices}. 
 \end{proof}

\subsection{Proof of Theorem~\ref{thm:centralizers_act_cocpctly}}

We will need the following fact about subgroups of solvable Lie groups. 

\begin{lemma}\cite[Proposition 3.8]{R1972}\label{lem:finite_generation} Let $G$ be a solvable Lie group with finitely many components and $H \leq G$ a closed subgroup. Let $H_0$ be the connected component of the identity in $H$. Then $H/H_0$ is finitely generated. 
\end{lemma}

For the rest of the section fix a naive convex co-compact triple $(\Omega, \Cc, \Lambda)$ and an Abelian subgroup $A \leq \Lambda$. Let $\overline{A}^{{\rm Zar}}$ be the Zariski closure in $\PGL_d(\Rb)$. Then $\overline{A}^{{\rm Zar}}$ is Abelian and has finitely many components. Since $A \leq \overline{A}^{{\rm Zar}}$ is discrete, Lemma~\ref{lem:finite_generation} implies that 
\begin{align*}
A = \ip{a_1,\dots, a_m}
\end{align*}
for some $a_1, \dots,a_m \in A$. In particular,
\begin{align*}
C_\Lambda(A) = \cap_{j=1}^m C_\Lambda(a_j).
\end{align*}

Next for $r > 0$ define 
\begin{align*}
M_{r} := \left\{ x \in \Cc : H_\Omega(x,a_jx) \leq r \text{ for all } 1\leq j \leq m\right\}.
\end{align*}

\begin{lemma} $C_{\Lambda}(A) \leq \Stab_{\Lambda}(M_r)$. \end{lemma}

\begin{proof} If $\gamma \in C_{\Lambda}(A)$ and $x \in M_{r}$, then 
\begin{align*}
H_\Omega(\gamma x, a_j \gamma x) = H_\Omega(\gamma x, \gamma a_j x) = H_\Omega(x,a_j x) \leq r
\end{align*}
Hence $\gamma x \in M_{r}$. So $\gamma M_r \subset M_r$. Applying the same argument to $\gamma^{-1}$ shows that $M_r \subset \gamma M_r$. 
\end{proof}

\begin{lemma}\label{lem:centralizer_one} For every $r > 0$, $C_{\Lambda}(A)$ acts co-compactly on $M_r$. \end{lemma}

The following argument comes from the proof of Theorem 3.2 in~\cite{R2001}.

\begin{proof} If $M_r= \emptyset$, then there is nothing to prove. So we may assume that $M_r \neq \emptyset$. 

Suppose for a contradiction that $C_{\Lambda}(A)$ does not act co-compactly on $M_r$. Fix some $x_0 \in M_r$. Then for each $n$ there exists some $x_n \in M_r$ such that 
\begin{align*}
H_\Omega\left(x_n, C_{\Lambda}(A) \cdot x_0 \right) \geq n.
\end{align*}
Since $\Lambda$ acts co-compactly on $\Cc$, there exist $M > 0$ and a sequence $\beta_n \in \Lambda$ such that 
\begin{align*}
H_\Omega(\beta_n x_0, x_n) \leq M.
\end{align*} 
for all $n \geq 0$. Then for $1 \leq j \leq m$
\begin{align*}
H_\Omega( \beta_n^{-1} a_j \beta_n x_0, x_0) 
&=H_\Omega( a_j \beta_n x_0, \beta_n x_0)\\
&  \leq H_\Omega(a_j \beta_n x_0, a_j x_n) +  H_\Omega(a_jx_n, x_n) + H_\Omega(x_n, \beta_n x_0) \\
& \leq M + r + M =  r+2M.
\end{align*}
Since $\Lambda$ acts properly on $\Omega$, for every $1 \leq j \leq m$ the set 
\begin{align*}
\{ \beta_n^{-1} a_j \beta_n : n \geq 0\}
\end{align*}
must be finite. So by passing to a subsequence we can assume that 
\begin{align*}
 \beta_n^{-1} a_j \beta_n  = \beta_1^{-1} a_j \beta_1
 \end{align*}
 for all $ 1\leq j \leq m$ and $n \geq 0$. Then $\beta_n \beta_1^{-1} \in \cap_{j=1}^m C_\Lambda(a_j)=C_\Lambda(A)$ for all $n \geq 0$. Then 
 \begin{align*}
n &\leq  H_\Omega\left(x_n, C_{\Lambda}(A) \cdot x_0 \right) \leq H_\Omega\left(x_n, \beta_n \beta_1^{-1} x_0 \right)\\
&  \leq H_\Omega\left(x_n, \beta_n x_0 \right) + H_\Omega\left(\beta_n x_0, \beta_n \beta_1^{-1} x_0 \right)\\
&  \leq M + H_\Omega\left(x_0, \beta_1^{-1} x_0 \right)
\end{align*}
for all $n \geq 0$, which is a contradiction. Hence $C_\Lambda(A)$ acts co-compactly on $M_r$. 
\end{proof}

\begin{lemma}\label{lem:centralizer_two} For any $r > 0$, 
\begin{align*}
{\rm ConvHull}_\Omega \left( M_r\right) \subset M_{2^{d-1}r}.
\end{align*}
\end{lemma}

\begin{remark} A similar estimate is established in~\cite[Lemma 8.4]{CLT2015}.\end{remark}

\begin{proof}
For $n \geq 0$, let $C_n \subset {\rm ConvHull}_\Omega \left( M_r \right)$ denote the elements which can be written as a convex combination of $n$ elements in $M_r$. Then $C_1 = M_r$ and by Carath\'eodory's convex hull theorem, $C_{d} =  {\rm ConvHull}_\Omega \left( M_r \right)$. We claim by induction that 
\begin{align*}
C_n \subset M_{2^{(n-1)}r}
\end{align*}
for every $1 \leq n \leq d$.

By definition $C_1 = M_r$ so the base case is true. Now suppose that 
\begin{align*}
C_n \subset M_{2^{(n-1)}r}
\end{align*}
and $p \in C_{n+1}$. Then there exists $p_1, p_2 \in C_n$ such that $p \in [p_1,p_2]$. Let $\sigma : [0,T] \rightarrow \Cc$ be the unit speed projective line geodesic with $\sigma(0)=p_1$ and $\sigma(T) = p_2$. Then $p=\sigma(t_0)$ for some $t_0 \in [0,T]$. Next for $1 \leq j \leq m$ let $\sigma_j = a_j \circ \sigma$. Then Lemma~\ref{lem:Crampons_dist_est} implies that
\begin{align*}
H_\Omega(p,a_j p) 
&= H_\Omega(\sigma(t_0), \sigma_j(t_0)) \leq H_\Omega(\sigma(0), \sigma_j(0))+H_\Omega(\sigma(T), \sigma_j(T)) \\
& = H_\Omega(p_1,a_jp_1) + H_\Omega(p_2, a_jp_2) \leq 2^{(n-1)}r + 2^{(n-1)}r = 2^n r
\end{align*}
Since $p \in C_{n+1}$ was arbitrary, we have
\begin{align*}
C_{n+1} \subset M_{2^{n}r}
\end{align*}
and the proof is complete.
\end{proof}

Combining Lemma~\ref{lem:centralizer_one} and Lemma~\ref{lem:centralizer_two} we have the following. 

\begin{lemma}\label{lem:co_cpct_on_convex_hull} For any $r > 0$, $C_\Lambda(A)$ acts co-compactly on ${\rm ConvHull}_\Omega \left( M_r \right)$. \end{lemma}

\begin{proof} Lemma~\ref{lem:centralizer_one} implies that $C_\Lambda(A)$ acts co-compactly on $M_{2^{d-1}r}$ and Lemma~\ref{lem:centralizer_two} implies that  ${\rm ConvHull}_\Omega \left( M_r \right)$ is a subset of $M_{2^{d-1}r}$. Then, since ${\rm ConvHull}_\Omega \left( M_r \right)$ is a closed $C_\Lambda(A)$-invariant subset of $M_{2^{d-1}r}$, the action of $C_\Lambda(A)$ on ${\rm ConvHull}_\Omega \left( M_r \right)$ is co-compact.
\end{proof}

\begin{lemma} $\Min_{\Cc}(A) \neq \emptyset$ and $C_\Lambda(A)$ acts co-compactly on ${ \rm ConvHull}_\Omega( \Min_{\Cc}(A))$.  \end{lemma}

\begin{proof} If $r > \max_{1 \leq j \leq d} \tau(a_j)$, then 
\begin{align*}
\Min_{\Cc}(A) = \cap_{a \in A} \Min_{\Cc}(a) \subset \cap_{j=1}^m \Min_{\Cc}(a_j) \subset M_r.
\end{align*}
So ${ \rm ConvHull}_\Omega( \Min_{\Cc}(A))$ is a closed $C_\Lambda(A)$-invariant subset of ${\rm ConvHull}_\Omega \left( M_r \right)$. 
Further, Lemma~\ref{lem:co_cpct_on_convex_hull} implies that $C_\Lambda(A)$ acts co-compactly on ${\rm ConvHull}_\Omega \left( M_r \right)$. So $C_\Lambda(A)$ also acts co-compactly on ${ \rm ConvHull}_\Omega( \Min_{\Cc}(A))$.   

Next we show that $\Min_{\Cc}(A) \neq \emptyset$. Pick $A^\prime \geq A$ a maximal Abelian subgroup in $\Lambda$. Then $A^\prime = C_{\Lambda}(A^\prime)$. By Lemma~\ref{lem:finite_generation} and the discussion following the lemma
\begin{align*}
A^\prime = \ip{a_1^\prime,\dots, a_n^\prime}
\end{align*}
for some $a_1^\prime, \dots, a_n^\prime \in A^\prime$. Notice that 
\begin{align*}
\Min_{\Cc}(A^\prime) = \cap_{a \in A^\prime} \Min_{\Cc}(a) \subset \cap_{a \in A} \Min_{\Cc}(a) = \Min_{\Cc}(A)
\end{align*}
and so it is enough to show that $\Min_{\Cc}(A^\prime) \neq \emptyset$. 

For $r > 0$ define 
\begin{align*}
M_{r}^\prime := \left\{ x \in \Cc : H_\Omega(x,a_j^\prime x) \leq r \text{ for all } 1\leq j \leq n\right\}.
\end{align*}
Then for $r$ sufficiently large, $M_r^\prime \neq \emptyset$. Further, by applying Lemma~\ref{lem:co_cpct_on_convex_hull} to $A^\prime$, we see that $A^\prime$ acts co-compactly on the convex set 
\begin{align*}
\Cc^\prime := { \rm ConvHull}_\Omega(M_r^\prime) \subset \Cc.
\end{align*}

Then by Theorem~\ref{thm:abelian_cc_actions} there exists a properly embedded simplex $S \subset \Cc^\prime \subset \Cc$ where 
\begin{enumerate}
\item $A^\prime \leq \Stab_\Omega(S)$, 
\item $A^\prime$ acts co-compactly on $S$, and 
\item $A^\prime$ fixes each vertex of $S$. 
\end{enumerate}
Then Proposition~\ref{prop:min_set_inv_simplex} implies that 
\begin{align*}
S \subset \Min_{\Cc}(A^\prime)
\end{align*}
and hence $\Min_{\Cc}(A^\prime)$ is non-empty. 
\end{proof}

\section{Proof of Theorem~\ref{thm:max_abelian}}\label{sec:pf_of_thm_max_abelian}

Theorem~\ref{thm:max_abelian} is a straightforward consequence of Theorems~\ref{thm:abelian_cc_actions} and~\ref{thm:centralizers_act_cocpctly}. Suppose that $(\Omega, \Cc, \Lambda)$ is a naive convex co-compact triple and $A \leq \Lambda$ is a maximal Abelian subgroup. Since $A$ is a maximal Abelian subgroup, $A = C_{\Lambda}(A)$. Then Theorem~\ref{thm:centralizers_act_cocpctly} implies that $A$ acts co-compactly on the non-empty convex subset 
\begin{align*}
{\rm ConvHull}_\Omega\left(\Min_{\Cc}(A)\right) \subset \Cc.
\end{align*}

Then by Theorem~\ref{thm:abelian_cc_actions} there exists a properly embedded simplex 
\begin{align*}
S \subset {\rm ConvHull}_\Omega\left(\Min_{\Cc}(A)\right) \subset \Cc
\end{align*}
where 
\begin{enumerate}
\item $A \leq \Stab_\Omega(S)$, 
\item $A$ acts co-compactly on $S$, and 
\item $A$ fixes each vertex of $S$. 
\end{enumerate}

It remains to show that $A$ has a finite index subgroup isomorphic to $\Zb^k$ where $k = \dim S$. Consider $V:=\Spanset S$ and the homomorphism 
\begin{align*}
\varphi & : A \rightarrow \Aut(S) \leq \PGL(V) \\
\varphi& (a) = a|_V.
\end{align*}
By changing coordinates we can assume that $V = \Rb^{k+1} \times \{0\}$ and 
\begin{align*} 
S = \{ [x_1 : \dots : x_{k+1} : 0 : \dots : 0] : x_1,\dots, x_{k+1} > 0\}.
\end{align*}
Since $A$ fixes the vertices of $S$, $\varphi(A)$ is a subgroup of 
\begin{align*}
G&:=\Big\{ \left[{\rm diag} \big( a_1, \dots, a_{k+1} \big)\right] \in \PGL_{k+1}(\Rb) : a_1, \dots, a_{k+1}>0 \Big\} \cong (\Rb^{k},+).
\end{align*}

Notice that $\ker(\varphi)$ fixes every point of $S$ and hence, since $A$ acts properly discontinuously on $\Omega$, must be a finite group. By Selberg's lemma, there exists a torsion-free finite index subgroup $\Lambda_0 \leq \Lambda$. Then $\Lambda_0 \cap A \leq A$ has finite index. Further $(\Lambda_0 \cap A) \cap \ker \varphi = \{\id\}$ and so $\varphi|_{\Lambda_0 \cap A}$ is injective. 

Finally, 
\begin{align*}
\Lambda_0 \cap A \cong \varphi(\Lambda_0 \cap A) \leq G\cong (\Rb^{k},+)
\end{align*} 
is a uniform lattice since $\Lambda_0 \cap A$ acts co-compactly and properly discontinuously on $S$. So $\Lambda_0 \cap A  \cong \Zb^k$.

\bibliographystyle{alpha}
\bibliography{geom}

\end{document}